\documentclass{amsart}

\usepackage[utf8]{inputenc}
\usepackage[T1]{fontenc}
\usepackage{lmodern}
\usepackage[text={16cm,23cm},centering,vcentering]{geometry}
\usepackage[english]{babel}
\usepackage{romanbar}
\usepackage{anyfontsize}
\usepackage{microtype}
\usepackage{amsmath,amsfonts,amssymb,amsthm}
\usepackage{dsfont}
\usepackage{stmaryrd} 
\usepackage{relsize}
\usepackage[svgnames]{xcolor}
\usepackage{mathrsfs}
\usepackage{mdframed}
\usepackage{subcaption}
\usepackage{enumitem}
\usepackage{ tipa }

\usepackage{graphicx}

\usepackage{float}
\usepackage{framed}

\usepackage{multirow,bigdelim}

\usepackage{hyperref}
\hypersetup{
    colorlinks=true, 
    linktoc=all,     
    colorlinks,
    citecolor=black,
    filecolor=black,
    linkcolor=black,
    urlcolor=black
}

\theoremstyle{definition}
\newtheorem{theoreme}{Theorem}
\newtheorem{prop}[theoreme]{Proposition}
\newtheorem{lem}[theoreme]{Lemma}
\newtheorem{cor}[theoreme]{Corollary}
\newtheorem{defi}[theoreme]{Definition}
\newtheorem{ex}[theoreme]{Example}
\newenvironment{rem}{\emph{Remark:}}{}

\newtheorem{thmB}{Conjecture}

\newcommand{\Sum}{\mathlarger{\sum}}

\newcommand{\gcdd}[2] {\operatorname{gcd}\left({#1},{#2}\right)}

\newcommand{\Z} {\mathbb{Z}}
\newcommand{\N} {\mathbb{N}}

\newcommand{\Q} {\mathbb{Q}}
\newcommand{\F} {\mathbb{F}}

\newcommand{\ens}{ \ \vert \ }

\usepackage{etoolbox}
\AtBeginEnvironment{enumerate}{\strut}
\AtBeginEnvironment{itemize}{\strut}

\newcommand{\noleftdelimiter}{\left.\kern-\nulldelimiterspace}

\counterwithin{theoreme}{section}
\begin{document}
\author{Alexis Lucas}

\address{
Normandie Université, 
Université de Caen Normandie - CNRS, 
Laboratoire de Mathématiques Nicolas Oresme (LMNO), UMR 6139,  
14000 Caen, France.
}
\email{alexis.lucas@unicaen.fr}
\title{ Wieferich and Mersenne primes for function fields}

\date{\today}


\begin{abstract}
  We study properties of recently introduced Wieferich primes for  Drinfeld modules, as their relation with Fermat equations and finitess or non-finiteness of their number. We also introduce Mersenne numbers for Drinfeld modules, and study the links between these two notions.
\end{abstract}

\maketitle
\tableofcontents

\section{Introduction}
We know from Fermat's little theorem that for any prime number $p$ and any integer $a$ not divisible by $p$, we have 
$$a^{p-1}\equiv 1 \pmod{p^2}.$$ We say that $p$ is Wieferich in base $a$ if 
$$a^{p-1}\equiv 1\pmod{p^2}.$$
This notion appeared in 1909 in the work of A. Wieferich \cite{wieferich1909letzten}, who showed that if there exists a prime number $p$ such that Fermat's equation 
$$x^p+y^p=z^p$$ has a solution $(x,y,z)\in \Z^3$ with $p\nmid xyz$, then $p$ is Wieferich in base $2$. This notion has been studied in this context of ``elementary arithmetic'', but has also been developed in the context of elliptic curves \cite{voloch2000elliptic}, and more recently in the context of number fields \cite{fellini2025wieferichprimesnumberfields}. As might be expected from the analogies between number fields and function fields or Drinfeld modules and elliptic curves, this notion was developed by D. Thakur \cite{thakurwieferich} for the Carlitz module, also studied in \cite{bamunoba2017note,quan2014carlitz}.Recently introduced by the author, X. Caruso and Q. Gazda for general Drinfeld modules \cite{cgl}, surprisingly, this notion has been linked to the special values of certain zeta functions of Drinfeld modules.

Another object that has been extensively studied is that of Mersenne numbers. In the classical case, a prime number is said to be a Mersenne prime if it is of the form $2^p-1$, where $p$ is a prime number. This concept was developed in the context of function fields for the Carlitz module in \cite{quan2014carlitz}.

In this paper, we introduce the notion of Mersenne primes for general Drinfeld modules, and discuss the purely arithmetic properties of these two notions of Wieferich primes and Mersenne primes.\\

Let us give more precise statements of our results.

Let $\theta$ be an indeterminate, $q=p^e$ be the power of a prime number $p$, $\F_q$ be the finite field with $q$ elements, and $A=\F_q[\theta]$ be the polynomial ring in the variable $\theta$ whose fraction field is denoted by $K=\F_q(\theta)$.

Let $\tau:K\rightarrow K, a\mapsto a^q$ be the Frobenius map. We then consider $A\{\tau\}$ to be the ring of  twisted polynomials in the variable $\tau$, subject to the multiplication rule 
$$\tau a=a^q\tau, \forall a\in A.$$

A Drinfeld $A$-module defined over $A$ (shortly a Drinfeld module in the following) is an $\F_q$-algebra homomorphism $\phi:A\rightarrow A\{\tau\}$, completely determined by the image of $\theta$ that satisfies
$$\phi_\theta=\theta\tau^0+\Sum\limits_{i=1}^r a_i\tau^i,$$ with $a_r\neq 0$ and $r\geq 1$ called the rank of $\phi$.

In the following, we do not distinguish the notion of prime polynomial and irreducible polynomial.
For any irreducible unitary polynomial $P$ of $A$, we can consider the finite $A$-module $\phi(A/PA)$, which is equal to $A/PA$ as a set, but whose $A$-module structure is given by $\phi$:
$$a.m=\phi_a(m), \forall a\in A, m\in A/PA.$$ Let us then denote by $\operatorname{Fitt}_A\phi(A/PA)$ its Fitting ideal, which is a non-zero ideal of $A$, and by $g_{P,\phi}$ its monic generator. In particular, according to \cite{fitting1}, the Fitting ideal satisfies
$$\operatorname{Fitt}_A \phi(A/PA)\subseteq \operatorname{Ann}_A \phi(A/PA)=\{ m\in A \ens \phi_m(a)\in PA, \forall a\in A\},$$ from which we obtain 
$$\phi_{g_{P,\phi}}(a)= 0 \pmod{P}, \forall a\in A.$$ This is analogous to Fermat's little theorem for Drinfeld modules.

\begin{defi}We say that a monic prime $P$ is $\phi$-Wieferich in base $a\in A$ if
$$\phi_{g_{P,\phi}}(a)= 0 \pmod{P^2}.$$
We say that $P$ is super $\phi$-Wieferich in base $a$ if
$$\phi_{g_{P,\phi}}(a)=0\pmod{P^3}.$$
\end{defi}

In Section \ref{section:wieferich}, we first prove a weak version of A. Wieferich's theorem. If $P$ is such that there exists a solution $(x,y,z)\in A^3$ to the Fermat equation 
$$y^{q^{r\deg P}}\phi_P\left(\dfrac{x}{y}\right)=z^{q^{r\deg P}}$$ with $P\nmid xyz$, then $P$ is $\phi$-Wieferich in a certain base $a$, see Theorem \ref{th:fermatdrinfeld}. We then prove the non-finiteness of Wieferich primes using the ABC-theorem for function fields and under a strong assumption of squared divisors of some quantities. We then show that, if $P$ is super $\phi$-Wieferich in base $a$ for a finite number of $P$, then $P$ is not $\phi$-Wieferich in base $a$ for an infinite number of $P$, see Theorem \ref{th:granvilledrinfeld}.

In section \ref{sec:Mersenne}, we introduce $\phi$-Mersenne numbers in base $a$, which are polynomials of the form $\phi_P(a)$ with $P$ a prime polynomial, and study the relationship between their primality and $\phi$-Wieferich numbers in base $a$. In particular, assuming the primality of $g_{P,\phi}$ for an infinite number of $P$, we prove that there are infinitely many Mersenne numbers that are not prime polynomials.

\section{Wieferich primes for Drinfeld modules}\label{section:wieferich}
This section is devoted to the study of Wieferich primes in the context of Drinfeld modules.

\subsection{Wieferich primes: first properties}
We keep the notation from the introduction. In particular, $P$ is an irreducible monic polynomial of $A=\F_q[\theta]$, and $\phi$ is a Drinfeld module of rank $r$. We denote by $\deg a$ the degree in $\theta$ of a polynomial $a\in A$, with $\deg0=-\infty$. We will keep as a guiding example the case of the Carlitz module, studied in \cite{bamunoba2017note,quan2014carlitz,thakurwieferich}, for which we generalize many results.

We denote by 
$$\phi(A)_{\operatorname{Tors}}=\{a\in A\ens \exists c\in A-\{0\}, \phi_c(a)=0\}.$$
We have the following way to compute $g_{P,\phi}$:
$$g_{P,\phi}=\det_{\F_q[X]}\left(X-\phi_\theta \ens A/PA\right)_{|X=\theta}.$$
We deduce that $g_{P,\phi}$ is a monic polynomial of degree the $\F_q$-dimension of $A/PA$, that is equal to $\deg P$.
We denote in the following for all monic irreducible polynomial $P$:
$$P=g_{P,\phi}+r_{P,\phi}$$ where $r_{P,\phi}\in A$, $\deg r_{P,\phi}<\deg P.$ So Fermat's little theorem can be rewritten as
$$\phi_P(a)= \phi_{r_{P,\phi}}(a)\pmod{P}\ \forall a\in A,$$ and the condition for being $\phi$-Wieferich in base $a$ can be rewritten as
$$\phi_P(a)= \phi_{r_{P,\phi}}(a)\pmod{P^2}.$$

We have an equivalent characterization for a prime ideal $P$ to be Wieferich, see \cite{cgl}, which we recall here. We fix a polynomial $a\in A$.

\begin{defi}\label{def:pi}\label{def:wieferich}
For a given ideal $I=mA$ of $A$, where $m\in A$, we define
\[
\pi_a(\phi,m)=\pi_a(\phi;I) = \ker( A \to A/I, \, b \mapsto \phi_b(a) )=\{b\in A\ens \phi_b(a)\in I\}.
\]

\end{defi}

We recall the following properties that can be found in \cite[Proposition 2.8, Proposition 2.11]{cgl}. We consider the following hypothesis.
\begin{enumerate}[leftmargin=10ex,label=$(H)_{P}$]
\item\label{Hyp}: \, if $q=2$, then $\deg P>1$.
\end{enumerate}
\begin{prop}\label{prop:properties-of-pi} Let $P$ be a monic prime ideal of $A$ satisfying \ref{Hyp}.
\begin{enumerate}[label=$(\arabic*)$]
\item\label{point1} For any non-zero ideal $I$ of $A$, $\pi_a(\phi,I)$ is a non-zero ideal of $A$.
\item We have $\pi_a(\phi,I)=A$ if and only if $a\in I$.
\item\label{item:divide} $\operatorname{Fitt}_A(\phi(A/I))\subseteq \pi_a(\phi;I)$.
\item\label{item:tour} For all $k\geq 2$, 
  $\pi_a(\phi;P^k)$ is either equal to $\pi_a(\phi;P^{k-1})$ or equal to $P\pi_a(\phi;P^{k-1})$. Furthermore, if $\pi_a(\phi;P^k)=P\pi_a(\phi;P^{k-1})$, then $$\pi_a(\phi;P^{k+1})=P\pi_a(\phi;P^{k}).$$
\item We have $\pi_a(\phi,P^k)=\pi_a(\phi,P)$ for all $k\geq 1$ if and only if $a\in \phi(A)_{\operatorname{Tors}}$.
\item Let $k$ be the largest integer such that $\pi_a(\phi,P)=\pi_a(\phi,P^k)$. Then
$$v_P\left(\phi_{g_{P,\phi}}(a)\right)=k+1.$$

\end{enumerate}
\end{prop}

\begin{defi}We denote by $\mathscr{P}_{\phi,P,a}$ the monic generator of $\pi_a(\phi,P)$, which we call the Drinfeld annihilator of $P$ in base $a$. It is therefore characterized by
$$\phi_m(a)= 0\pmod{P}\Leftrightarrow \mathscr{P}_{\phi,P,a}\vert m,  m\in A.$$
\end{defi}
In particular, point \ref{item:divide} of Proposition \ref{prop:properties-of-pi} gives the following relation for any prime $P$ satisfying \ref{Hyp}:
$$\mathscr{P}_{\phi,P,a}\vert g_{P,\phi}, \forall a\in A,$$ and $P$ is $\phi$-Wieferich in base $a$ if and only if $\pi_a(\phi,P)=\pi_a(\phi,P^2)$.

For $a,b\in A$, we denote by $\gcdd{a}{b}$ the greatest common factor of $a$ and $b$. We introduce the following lemma that will be usefull in the following.
\begin{lem}\label{pgcddrinfeld} For all non-zero polynomials $u,v\in A$, we have 
$$\gcdd{\phi_u(a)}{\phi_v(a)}=\phi_{\gcdd{u}{v}}(a).$$
\end{lem}
\begin{proof}
If $u=Qd$ for polynomials $d,Q\in A$, then $\phi_u(a)=\phi_{Qd}(a)$, which is divisible by $\phi_d(a)$. We can deduce that
$$\phi_{\gcdd{u}{v}}(a) \vert \gcdd{\phi_u(a)} {\phi_v(a)}.$$ Let us set $m=\gcdd{\phi_u(a)} {\phi_v(a)}$, satisfying in particular that $m$ divides $\phi_u(a)$ and $m$ divides $\phi_v(a)$. We therefore have 
$$\mathscr{P}_{\phi,m,a}\vert u \text{ and }  \mathscr{P}_{\phi,m,a}\vert v,$$
from which we obtain 
$$\mathscr{P}_{\phi,m,a}\vert \gcdd{u}{v},$$ in other words
$$m\vert \phi_{\gcdd{u}{v}}(a).$$
\end{proof}

It has been proved, for the Carlitz module \cite[Theorem 4.3.15]{bamunoba2017note}, that if $q>2$, then $P$ is $C$-Wieferich in base $1$ if and only if it is in base $\theta$. We therefore wonder when we can replace the base $a$ with other bases. We have the following result.
\begin{lem}
Let $a\in A-\phi(A)_{\operatorname{Tors}}$ and $P$ be an irreducible monic polynomial satisfying \ref{Hyp}. Let $d\in A$ be such that $\gcdd{d}{\mathscr{P}_{\phi,P,A}}=1$. Then $P$ is $\phi$-Wieferich in base $a$ if and only if $P$ is $\phi$-Wieferich in base $\phi_d(a)$ (and there are an infinite number of them).
\end{lem}

\begin{proof} Let $a\notin \phi(A)_{\operatorname{Tors}}$. Let $d\in A$ be such that $\gcdd{d}{\mathscr{P}_{\phi,P,a}}=1$. Then, for all $x\in A$, we have
    $$\phi_{dx}(a)\in PA\Leftrightarrow\mathscr{P}_{\phi,P,a}\vert dx \Leftrightarrow \mathscr{P}_{\phi,P,a}\vert x.$$
    We therefore have for such $d$ that the element $b=\phi_d(a)$ (and these elements $\phi_d(a)$ are pairwise distinct since $a$ is not a torsion point)
    $$\pi_a(\phi,P)=\pi_b(\phi,P)$$
    so
    $$\mathscr{P}_{\phi,P,a}=\mathscr{P}_{\phi,P,b}.$$

\end{proof}

\begin{ex}
    We consider the Carlitz module $C$ again, and assume that $q>2$, so $C(A)_{\operatorname{Tors}}=\{0\}$. We consider $a=1$, and let $P$ be $C$-Wieferich in base $1$.

    We obtain that for any polynomial $d$ such that $(d,\pi_1(C,P))=1$, we have that $P$ is $C$-Wieferich in base $1$ if and only if $P$ is $C$-Wieferich prime in base $\phi_d(1)$. In particular, we have that $d=\theta-1$ is coprime with $\pi_1(d,P)$. Indeed, otherwise we would have $\pi_1(C,P)=\theta-1$, so $C_{\theta-1}(1)=0\pmod{P^2}$, but $C_{\theta-1}(1)=\theta$ is never divisible by $P^2$ for any $P$. Thus, we deduce that $P$ is $C$-Wieferich in base $1$ if and only if $P$ is $C$-Wieferich in base $C_{\theta-1}(1)=\theta$, which was proven in \cite[Theorem 4.6]{bamunoba2017note} using more computational methods.

\end{ex}

\subsection{The Fermat equation for Drinfeld modules}

Let $\phi:A\rightarrow A\{\tau\}$ be a Drinfeld $A$-module of rank $r$.

For any non-zero polynomial $N\in A-\{0\}$, we consider folliwing the homogeneous Fermat equation for $(\phi,N)$:
\begin{equation}\label{fermatN}
   y^{q^{r\deg N}}\phi_N\left(\dfrac{x}{y}\right)=z^{q^{r\deg N}}
\end{equation}
 that is a polynomial in $A[x,y,z]$.

In this section only, we slightly modify Definition \ref{def:wieferich} of $\phi$-Wieferich primes to extend Thakur's definition \cite{thakurwieferich}, also used in \cite{bamunoba2017note}.
\begin{defi}\label{def:wieferich 2}
We say that $P$ is $\phi$-Wieferich in base $a$ if $$\phi_{P}(a)= \phi_{r_{P,\phi}}(a)^{q^{r\deg P}}\pmod{P^2}.$$
\end{defi}
\begin{theoreme}\label{th:fermatdrinfeld} Consider $(x,y,z)\in A^3$ and let $P$ be a monic prime not dividing $xyz$, such that $(x,y,z)$ is a solution of the homogeneous Fermat equation for $(\phi,P)$. Then there exists $a\in A-PA$ such that $P$ is $\phi$-Wieferich in base $a$ (for Definition \ref{def:wieferich 2}).
\end{theoreme}

We draw inspiration from the proof of \cite[Theorem 3.2]{bamunoba2017note}. 

\begin{proof}Note first that if we consider $K_P$ the $P$-adic completion of $K$, and that if we denote by $A_P=\{a\in K_P\ens v_P(a)\geq 0\}$ the $P$-adic completion of $A$, then we have
$$\phi_{g_{P,\phi}}(a)\equiv 0\pmod{P}, \forall a\in A_P.$$

Let $(x,y,z)$ be a solution of the homogeneous Fermat equation for $(\phi,P)$, such that $P\nmid xyz$. We also consider $x_1=\dfrac{x}{z}$ and $y_1=\dfrac{y}{z}$, which are elements of $A_{P}$. 
   The homogeneous Fermat equation for $P$ can be rewritten as follows:
    $$y_1^{q^{r\deg P}}\phi_P\left(\dfrac{x_1}{y_1}\right)-1=0,$$ thus,
    \begin{equation}\label{eq:fermat2}
        y_1^{q^{r\deg P}}\phi_{g_{P,\phi}}\left(\dfrac{x_1}{y_1}\right)+y_1^{q^{r\deg P}}\phi_{r_{P,\phi}}\left(\dfrac{x_1}{y_1}\right)-1=0.
    \end{equation}
    By reducing modulo $P$, we obtain
    \begin{equation}\label{eq:fermat3}y_1\phi_{r_{P,\phi}}\left(\dfrac{x_1}{y_1}\right)=1\pmod{P}.\end{equation}
  We denote by
    $$\left\{\begin{aligned}& \phi_{r_{P,\phi}}\left(\dfrac{x_1}{y_1}\right)=\dfrac{x_1}{y_1}c, c\in K^\ast, v_P(c)=0, \\
    &x_1c=1+dP, d\in K^\ast, v_P(d)=0.\end{aligned}\right.$$
We want to reduce modulo $P^2$. First, remark that
    $$\phi_{g_{P,\phi}}\left(\dfrac{x_1}{y_1}\right)=\phi_{g_{P,\phi}}\left(\dfrac{1+dP}{y_1c}\right)=\phi_{g_{P,\phi}}\left(\dfrac{1}{y_1c}\right)-r_{P,\phi}\dfrac{d P}{y_1c}\pmod{P^2}$$ and
    $$\phi_{r_{P,\phi}}\left(\dfrac{x_1}{y_1}\right)=\phi_{r_{P,\phi}}\left(\dfrac{1+dP}{y_1c}\right)=\phi_{r_{P,\phi}}\left(\dfrac{1}{y_1c}\right)+r_{P,\phi}\dfrac{dP}{y_1c}\pmod{P^2}.$$
    By reducing Equation \eqref{eq:fermat2} modulo $P^2$, we obtain
$$y_1^{q^{r\deg P}}\phi_{g_{P,\phi}}\left(\dfrac{1}{y_1c}\right)-y_1^{q^{r \deg P}}r_{P,\phi}\dfrac{d P}{y_1c}+y_1^{q^{ r\deg P}}\phi_{r_{P,\phi}}\left(\dfrac{1}{y_1c}\right)+y_1^{q^{r\deg P}} r_{P,\phi}\dfrac{d P}{y_1c}-1=0\pmod{P^2}$$ so
$$y_1^{q^{r\deg P}}\phi_{g_{P,\phi}}\left(\dfrac{1}{y_1c}\right)+y_1^{q^{ r\deg P}}\phi_{r_{P,\phi}}\left(\dfrac{1}{y_1c}\right)-1=0\pmod{P^2}$$
thus,
$$(y_1c)^{q^{r \deg P}}\phi_{g_{P,\phi}}\left(\dfrac{1}{y_1c}\right)= c^{q^{r\deg P}} -(cy_1)^{q^{r\deg P}}\phi_{r_{P,\phi}}\left(\dfrac{1}{y_1c}\right)\pmod{P^2}$$ then
$$\phi_{g_{P,\phi}}\left(\dfrac{1}{y_1c}\right)=\left(\dfrac{1}{y_1}\right)^{q^{r\deg P}}-\phi_{r_{P,\phi}}\left(\dfrac{1}{y_1c}\right)\pmod{P^2}.$$
By Equation \eqref{eq:fermat3}, we have
$$\begin{aligned}\left(\dfrac{1}{y_1}\right)^{q^{r\deg P}}&= \phi_{r_{P,\phi}}\left(\dfrac{x_1}{y_1}\right)^{q^{r\deg P}} \pmod{P^2}\\
&= \phi_{r_{P,\phi}}\left(\dfrac{x_1c}{y_1c}\right)^{q^{r\deg P}}\pmod{P^2}\\
&= \phi_{r_{P,\phi}}\left(\dfrac{1+dP}{y_1c}\right)^{q^{r\deg P}}\pmod{P^2}\\
&= \phi_{r_{P,\phi}}\left(\dfrac{1}{y_1c}\right)^{q^{r\deg P}} \pmod{P^2}.\end{aligned}$$

By choosing $a\in A-PA$ such that $a=\dfrac{1}{cy_1}\pmod{P^2}$, we have 
$$\phi_{g_{P,\phi}}(a)=\phi_{r_{P,\phi}}(a)^{q^{r\deg P}}-\phi_{r_{P,\phi}(a)}\pmod{P^2}$$ so
$$\phi_P(a)= \phi_{r_{P,\phi}}(a)^{q^{r\deg P}}\pmod{P^2}.$$
\end{proof}

Note that in the classical case, A. Wieferich \cite{wieferich1909letzten} proved that if the first case of Fermat's equation for a prime number $p$ had a solution, then $p$ is a Wieferich prime in base $a=2$. The previous result is therefore a very weak version of A. Wieferich's theorem, since it does not specify any base.

\subsection{ The number of Wieferich primes}
Let $a\in A-\phi(A)_{\operatorname{Tors}}$. In this Subsection, we look at the finitude or infinitude of $\phi$-Wieferich primes in base $a$. 

Note that in the classical case, Silvermann \cite{silverman1988wieferich} has shown that, under the ABC-conjecture, there are infinitely many prime numbers that are not Wieferich numbers. In the context of function fields, we have such a conjecture, which in this context is a theorem, stated as follows.

For $a\in A$, let $\operatorname{rad}(a)=\prod\limits_{P\vert a} P$ be the product of the distinct primes dividing $a$.

\begin{theoreme}\label{th:abc}[Mason–Stothers] Let $a,b,c \in A$ be pairwise coprime to each other and such that the derivative of at least one of the three polynomials is non-zero. Then we have the inequality:

$$\max \{ \deg a, \deg b, \deg c\} \leq \deg(\operatorname{rad}(abc))-1.$$
\end{theoreme}

Unfortunately, this theorem does not seem to be sufficient to prove the non-infinity of $\phi$-Wieferich primes in base $a$. We prove it under the following conjecture.

\begin{thmB}\label{conj:abcsanscarré}
   There exists infinitely many $b\in a$ such that $\dfrac{\phi_b(a)}{a}$ has no square factors and has a non-zero derivative (at $\theta$).
\end{thmB}
\begin{theoreme}
   Assume Conjecture \ref{conj:abcsanscarré} to be true. Then there exists an infinite number of non-$\phi$-Wieferich primes in base $a$, for any $a\in A-\phi(A)_{\operatorname{Tors}}$.
\end{theoreme}
\begin{proof} Let us assume Conjecture \ref{conj:abcsanscarré} to be true.
For all $b\in A-\{0\}$, we denote by 
$$\dfrac{\phi_b(a)}{a}=u_{a,b}v_{a,b}$$ with 
$$u_{a,b}=\prod\limits_{P\vert \frac{\phi_b(a)}{a}, P^2\nmid \frac{\phi_b(a)}{a}} P, \text{ and } v_{b,a}=\prod\limits_{P^2\vert \frac{\phi_b(a)}{a}} P^{v_P\left(\dfrac{\phi_b(a)}{a}\right)}.$$
 In particular, we have
$$\deg\left(\operatorname{rad}(v_{a,b})\right)\leq \dfrac{\deg(v_{a,b})}{2}.$$
Any prime $P$ dividing $u_{a,b}$ satisfies $P^2$ does not divide $\phi_b(a)$, in other words, for these $P$, we have $\pi_a(\phi,P)\neq \pi_a(\phi,P^2)$, so they are not $\phi$-Wieferich in base $a$.

Consider
$$S=\left\{ b\in A\ens \dfrac{\phi_b(a)}{a} \text{ is square-free and has non-zero derivative} \right\},$$ which is an infinite set according to Conjecture \ref{conj:abcsanscarré}. Consider also
$$W=\left\{P \text{ prime in }A \ens P \text{ is not } \phi\text{-Wieferich in base } a\right\}.$$
Suppose, by contradiction, that $W$ is a finite set. 

In particular, the set $\{u_{a,b}, b\in A\}$ is a finite set, so the set $\{\deg u_{b,a}, b\in A\}$ is bounded by a constant $C\geq 0$.

Let us first show that the set 
$$T=\left\{ v_{a,b}, b\in A \ens b-1\in S\right\}$$ is infinite. If not, then we would have that the set $\{P \text{ prime} \ens P\vert v_{b,a}, b-1\in S\}$ is finite, in other words the set $L=\{P \text{ prime}\ens P^2\vert \phi_b(a), b-1\in S\}$ is finite. Since the set $S$ is infinite, the set $\{\phi_b(a), b-1\in S\}$ is infinite, so the set 
$$\{P \text{ prime} \ens P\vert \phi_b(a), b-1\in S\}$$ is infinite, and since the set $L$ is finite, we would have that the set
$$\{P \text{ prime }\ens P\parallel \phi_b(a), a-1\in S\}$$ is infinite, i.e. the set
$$\{u_{b,a}, a-1\in S\}$$ is infinite, which contradicts the finiteness of $W$. Therefore, the set $T$ is infinite.

 We now apply Mason–Stothers' theorem to $\left(\dfrac{\phi_b(a)}{a},1, \dfrac{\phi_{b-1}(a)}{a}\right)$ where $b-1$ runs through the set $S$ which is infinite (and we have $v_{a,b-1}=1$ in this case). We have
$$\begin{aligned}\deg (v_{a,b})&\leq \deg \dfrac{\phi_a(b)}{a}\\
&\leq \max\left( \deg \dfrac{\phi_b(a)}{a}, 1, \deg \dfrac{\phi_{b-1}(a)}{a}\right),\\
&\leq \deg(\operatorname{rad} v_{a,b} v_{a,b-1} u_{a,b} u_{a,b-1})-1 \text{ by Theorem \ref{th:abc}},\\
&\leq \deg(\operatorname{rad} v_{a,b})+\deg(\operatorname{rad} v_{a,b-1})+\deg(\operatorname{rad} u_{a,b})+\deg(\operatorname{rad} u_{a,b-1})-1\\
&\leq \dfrac{\deg(v_{a,b})}{2}+0+C+C-1,\end{aligned}$$

thus, the sequence $(\deg(v_{a,b}))_{b-1\in S}$ is bounded, which contradicts the non-finiteness of the set T.

\end{proof}
\begin{rem}
    It has been shown \cite[Lemma 3.5]{angles2012arithmetic} that for $q>2$, $\phi=C$ the Carlitz module and $a=1$, there exist infinitely many primes that are not $C$-Wieferich in base $a$, without using these conjectures. The proof od this result does not apply at all to the case of general Drinfeld modules.
\end{rem}

As a partial converse of the previous result, we have the following result.

\begin{prop} Assume that there exists a finite number of polynomials $b\in A$ such that $\dfrac{\phi_b(a)}{a}$ is square free. Then there exists an infinite number of primes that are $\phi$-Wieferich in base $a$.
\end{prop}

\begin{proof}
    Firstly, by Lemma \ref{pgcddrinfeld}, for all distinct prime polynomials $P_1,P_2$, we have 
    $$\gcdd{\dfrac{\phi_{P_1}(a)}{a}}{\dfrac{\phi_{P_2}(a)}{a}}=1.$$
   Assume that there exists a finite number of polynomials $b\in A$ such that $\dfrac{\phi_b(a)}{a}$ is square free. Then the set 
   $$K=\left\{P \text{ monic prime} \ens \dfrac{\phi_P(a)}{a} \text{ has a square factor} \right\}$$
   is infinite. Thus, the set
   $$L=\left\{ Q \text{ prime} \ens Q^2\vert\dfrac{\phi_P(a)}{a}, P\in  K\right\}$$ is infinite.
   But for any $Q\in L$ not dividing $a$ and dividing some $\dfrac{\phi_P(a)}{a}$, we have 
   $$\mathscr{P}_{\phi,Q,a}\vert P $$ so $$\mathscr{P}_{\phi,Q,a}=P$$ since $Q$ does not divide $a$. We then have $\phi_{\mathscr{P}_{\phi,Q,a}}(a)\in Q^2A$, thus $Q$ is $\phi$-Wieferich in base $a$.
\end{proof}

The purpose of the remainder of this section is to prove the following theorem, which is the analogue for Drinfeld modules of an unconditional theorem by Granville \cite{granville1985refining}.

\begin{theoreme}\label{th:granvilledrinfeld} Let $\phi$ be an $A$-Drinfeld module defined on $A$, and let $a\in A-\phi(A)_{\operatorname{Tors}}$. Assume that for all but a finite number of primes $P$, we have
    $\pi_a(\phi,P)=\pi_a(\phi,P^2)$. Then there exist infinitely many primes $P$ such that 
    $$\pi_a(\phi,P^3)=\pi_a(\phi,P^2).$$
\end{theoreme}

In other words , if $P$ is super $\phi$-Wieferich in base $a$ for a finite number of primes $P$, then there exists an infinite number of primes that are not $\phi$-Wieferich in base $a$.

We will need the following lemma.

\begin{lem}\label{lem:mordell} Let $\phi$ be a Drinfeld module of rank $r$. Then for all $a\in A-\{0\}$, the following equation, with variables $x$ and $y$, 
$$C_{\theta^3}(x)+a=ay^2$$
has a finite number of solutions in $A$.
\end{lem}
\begin{proof}
     The polynomial $\phi_{\theta^3}(x)$ is of the form 
$$\phi_{\theta^3}(x)=\theta^3 x+\Sum\limits_{i=1}^{3r} a_{i,\theta^3} x^{q^i}\in A[x]$$ therefore it has simple roots and degree (in $x$) $q^{3r}>4$. According to \cite[Lemma 6.2.2]{stichtenoth2009algebraic}, the curve defined by 
$$\phi_{\theta^3}(x)+a=ay^2$$ is a hyperelliptic curve of genus 
$$g_\phi=\left\{\begin{aligned} &\dfrac{q^{3r}-1}{2} \text{ if $q$ is odd}, \\
& \dfrac{q^{3r}-2}{2} \text{ if $q$ is even},\end{aligned}\right.$$
and therefore of genus $g_\phi>2$. According to Samuel-Voloch's theorem, see \cite{samuel1966lectures,voloch1991diophantine}, we obtain the result.
\end{proof}
\begin{proof}[Proof of Theorem \ref{th:granvilledrinfeld}]
    Suppose, by contradiction, that $\pi_a(\phi,P)=\pi_a(\phi,P^2)$ for an infinite number of $P$ and that 
    $$\pi_a(\phi, P^2)=\pi_a(\phi,P^3)$$ for a finite set $E_1$ of primes.
        Let $E_2$ be the finite set of primes $P$ such that $P$ divides $a$. In particular, if $P\notin E_2$, then $\pi_a(\phi,P)\neq A$. Let $E=E_1\cup E_2$ that is a finite set.

Let $Q$ be a prime and $P\notin E$ such that $P\vert \phi_{Q}(a)$. We therefore have $Q\in \pi_a(\phi,P)=\pi_a(\phi,P^2)$ and thus $P^2\vert \phi_Q(a)$. Furthermore, we have that $P^3\nmid \phi_Q(a)$, otherwise we would have $\phi_Q(a)\in \pi_a(\phi,P^3)=P\pi_a(\phi,P^2)$ and therefore we would have that $P$ strictly divides $Q$, which is absurd.

Thus, for any prime $Q$ and $P\notin E$ dividing $C_Q(a)$, we have 
$$v_P(C_Q(a))=2.$$ 
If we denote by $$S=\left\{ Q \text{ primes }\ens \text{ the prime divisors of } \dfrac{\phi_Q(a)}{a} \text{ do not belong to } E_1\right\},$$ then  we have for all $Q\in S$
$$\phi_Q(a)=a R_Q^2,$$ with $R_Q\in A$ a non-zero polynomial.

Next, for all distinct prime numbers $Q_1$ and $Q_2$, we have by Lemma \ref{pgcddrinfeld}
\begin{equation}\label{eq:pgcdcarlitz}(\phi_{Q_1}(a), \phi_{Q_2}(a))=a.\end{equation}

Since the sequence $(\phi_Q(a))_{Q \text{ prime}}$ is infinite (and the terms are pairwise prime to each other), and since $E_1$ is finite, we first deduce that $S$ is infinite.

Furthermore, by Equation \eqref{eq:pgcdcarlitz}, the polynomials $R_Q(a)$ are pairwise distinct (even coprime between themselves) for $Q\in S$. Moreover, according to Dirichlet's theorem, we know that the set of prime numbers $Q$ such that 
$$Q= 1 \pmod{\theta^3}$$
is infinite. Combined with Equation \eqref{eq:pgcdcarlitz}, we deduce that the following set is infinite:
$$W=\left\{Q \text{ primes} \ens Q= 1\pmod{\theta^3} , \text{ the prime divisors $P$ of $\dfrac{\phi_Q(a)}{a} $ do not live in $E_1$}\right\}.$$
For all $Q\in W$, if we write $x_Q=\phi_{\frac{Q-1}{\theta^3}}(a)\in A$, which are pairwise distinct since $a\notin \phi(A)_{\operatorname{Tors}}$,then we have 
$$\phi_{\theta^3(x_Q)}= \phi_{Q-1}(a)=\phi_Q(a)-a=aR_Q^2-a.$$
Thus, the equation
$$a+\phi_{\theta^3}(x)=a y^2$$ has infinitely many solutions in $A$, which contradicts Lemma \ref{lem:mordell}.
\end{proof}

Note that for all $k\geq 2$, the curve $$ C_{3,k}: a+\phi_{\theta^3}(x)=ay^k$$ has a finite number of points. Indeed, on the one hand, if $k=2l$ is even, then we reduce to the curve in Lemma \ref{lem:mordell} by setting $Y=y^l$, which has a finite number of integer points.

On the other hand, if $k$ is odd, then the curve $C_{3,k}$ is non-isotrivial and of genus $g$ given by (see \cite[Proposition 3.7.3]{stichtenoth2009algebraic})
$$g=\left\{ \begin{aligned} &\dfrac{(k-1)(q^{3r-1})}{2} \text{ if } \gcdd{p}{k}=1, \\ 
&\dfrac{(p-1)(q^{3r-1})}{2} \text{ otherwise, that is if } p\vert k,\end{aligned}\right.$$
so has genus $>2$. According to Samuel-Voloch's theorem, it admits a finite number of integer points. Then the previous proof generalizes trivially to obtain the following result.

\begin{cor}Let $\phi$ be a Drinfeld $A$-module defined on $A$, let $k\geq 2$ be an integer and $a\in A-\phi(A)_{\operatorname{Tors}}$. Assume that for all but a finite number of primes $P$, we have
    $\pi_a(\phi,P)=\pi_a(\phi,P^k)$. Then, there are infinitely many primes $P$ such that 
    $$\pi_a(\phi,P)=\pi_a(\phi,P^{k+1}).$$
\end{cor}
The previous result tells us that, if there exists a rank $k\geq 2$ such that $\pi_a(\phi,P^k)= \pi_a(\phi,P^{k+1})$ for a finite number of monic irreducible polynomials $P$, then there exists an infinite number of $P$ such that $\pi_a(\phi,P^{k-1})\neq \pi_a(\phi,P^k)$.

\section{Mersenne numbers for Drinfeld modules}\label{sec:Mersenne}

The purpose of this section is to introduce Mersenne numbers in the context of Drinfeld modules, and to relate them to the annihilator polynomials $\mathscr{P}_{\phi,P,a}$ and to Wieferich primes. Recall that in this work, primes does not mean monic irreducible but only irreducible.

\begin{defi}
Let $a\in A-\{0\}$ and $\phi$ be a Drinfeld module of rank $r$.
\begin{enumerate}
    \item A $\phi$-Mersenne number in base $a$ is a polynomial of the form $ \phi_P(a)$ with $P$ a prime polynomial.
\item 
   A prime $\phi$-Mersenne number in base $a$ is a prime polynomial of the form $ \phi_P(a)$ with $P$ a prime polynomial. 
    \end{enumerate}
\end{defi}

\begin{lem}\label{lem:mersennepremier}
    Let $a\in A$. Then for any prime $P$, if $\phi_P(a)$ is prime, then we have one of the following two cases:
    \begin{enumerate}
        \item either $a$ is prime and $a\in \phi(A)_{\operatorname{Tors}}$,
        \item or $a\in \F_q^\ast$.
    \end{enumerate}
\end{lem}

\begin{proof}
     We know that $a$ divides $\phi_P(a)$. Thus, on the one hand, for $\phi_P(a)$ to be prime, we must have $a$ to be a prime polynomial or $a\in \F_q^\times$. 
     Suppose that $a$ is a prime polynomial, and denote by $\phi_P(a)=ba$ where $b\in A$. We then have the following implications:
     $$\phi_P(a) \text{ prime } \Rightarrow b\in \F_q^\times\Rightarrow \phi_{P-b}(a)=0\Rightarrow a\in \phi(A)_{\operatorname{Tors}}.$$
\end{proof}

We have the following lemma that gives a characterization of the primality of annihilator polynomials $\mathscr{P}_{\phi,Q,a}$.
\begin{lem}
   Let $a\in A$ and let $Q$ be a prime in $A$ that does not divide $a$. Then we have the following equivalence:
    $$\mathscr{P}_{\phi,Q,a} \text{ is prime} \Leftrightarrow Q\vert \phi_P(a) \text{ for some monic irreducible polynomial $P$}. $$
\end{lem}

\begin{proof}
    Assume first that $\mathscr{P}_{\phi,Q,a}=P $ is a prime polynomial. We then have
    $$\phi_P(a)=\phi_{\mathscr{P}_{\phi,Q,a}}(a)= 0\pmod{Q}.$$

    Conversely, suppose that $Q$ divides $\phi_P(a)$ for a certain irreducible polynomial $P$. We therefore have $\mathscr{P}_{\phi,Q,a}\vert P$, so $\mathscr{P}_{\phi,Q,a}=1$ or $\mathscr{P}_{\phi,Q,a}=P$. But since $Q$ does not divide $a$, we have $\mathscr{P}_{\phi,Q,a}\neq 1$, so 
    $$\mathscr{P}_{\phi,Q,a}=P.$$
\end{proof}

   \begin{prop}\label{prop:mersennepaswieferich}
Assume that $a\notin \phi(A)_{\operatorname{Tors}}$. Let $M_P= \phi_P(a)$ be a $\phi$-Mersenne prime $Q$ in base $a$. Then $Q=M_P$ is not $\phi$-Wieferich in base $a$. Moreover, if $P$ satisfies \ref{Hyp} and $a\in \phi(A)_{\operatorname{Tors}}$, then $M_P=\phi_P(a)$ is not a prime polynomial.
\end{prop}

\begin{proof}
Assume first that $a\in \phi(A)_{\operatorname{Tors}}$. Set $P_0=\beta P$ with $\beta \in \F_q^\ast$ such that $P_0$ is a monic prime. Since $(q,\deg(P))\neq (2,1)$, we have by \cite[Corollary 2.6.5]{lucas:tel-05296002}
$$\phi_{g_{P_0,\phi}}(a)=0.$$
Assume, for contradiction, that $\phi_P(a)=Q$ is a prime polynomial, thus $\phi_{P_0}(a)=\beta Q:=Q_0$ is a prime. Then we have $\mathscr{P}_{\phi, Q_0, a}=P_0$. Moreover, 
$$\phi_{P_0}(a)=\phi_{P_0-g_{P_0,\phi}}(a)=\phi_{r_{P_0,\phi}}(a)=Q.$$
It follows that
$$P_0\vert \mathscr{P}_{\phi, Q_0,a}\vert r_{P_0,\phi},$$
hence $r_{P_0,\phi}=0$, so $\mathscr{P}_{\phi,Q_0,a}=0$, a contradiction.

Now let $a\notin \phi(A)_{\operatorname{Tors}}$ be such that $M_P= \phi_P(a)$ is a $\phi$-Mersenne $Q$ in base $a$. Then, 
$$P\in \pi_a(\phi,Q)-\pi_a(\phi,Q^2),$$ that is, $Q$ is not $\phi$-Wieferich in base $a$.
   
\end{proof}

The goal in the remainder of this section is to prove the following theorem.

\begin{theoreme}\label{th:mersenne} Let $a\in A$ and $\phi$ be a Drinfeld module.
There exist infinitely many primes $P$ such that $\phi_P(a)$ is not prime; in other words, there exist infinitely many 
$\phi$-Mersenne numbers in base $a$ which are not prime.
\end{theoreme}

By Proposition \ref{prop:mersennepaswieferich}, it suffices to prove the result for $a\in A-\phi(A)_{\operatorname{Tors}}$, and by Lemma \ref{lem:mersennepremier}, it suffices to prove it for $a\in \F_q^\times$. We prove it without loss of generality for $a=1$ satisfying $1\notin \phi(A)_{\operatorname{Tors}}$.

Unfortunately, we are only able to prove Theorem \ref{th:mersenne} by assuming the following conjecture.

\begin{thmB}\label{conj} Let $\phi:A\rightarrow A\{\tau\}$ be a Drinfeld module.
There exist infinitely many prime polynomials $P$ such that $g_{P,\phi}$ is prime.
\end{thmB}

\begin{rem}
    The previous Conjecture \ref{conj} is the analog of a question studied by Koblitz for elliptic curves, see \cite{koblitz1988primality}, that is as follows. Let $E$ be an elliptic curve defined over $\Q$. Are there infinitely prime numbers $p$ such that $\vert E(\F_p)\vert$ is prime ? 
\end{rem}

 We begin by stating the following lemma.

\begin{lem}\label{lem:degré} Let $\phi:A\rightarrow A\{\tau\}$ be a Drinfeld module of rank $r\geq 1$, such that $1\notin \phi(A)_{\operatorname{Tors}}$. Then there exists a constant $c_\phi>0$ such that
$$\forall a\in A, \deg a>c_\phi \Rightarrow \deg \phi_a(1)>\deg a.$$
\end{lem}

\begin{proof}

Write $\phi_\theta=\sum\limits_{i=0}^r a_i\tau^i\in A\{\tau\}$ with $a_r\neq 0$ and $a_0=0$.
Set
$$M=\max_{i=0,\ldots, r-1} \dfrac{\deg(a_i)}{q^r-q^i}>0.$$
Let $b\in A$ be such that $\deg b>M$. Then
\begin{equation}\label{eq:degré}\deg \phi_\theta(b)=q^r\deg b+\deg a_r>\deg(b).\end{equation}
Since $1\notin \phi(A)$, the set $\{\phi_{\theta^n}(1), n\in \N\}$ is infinite, and there exists a minimal integer $N$ such that $\deg \phi_{\theta^N}(1)>M$.
By Equation \eqref{eq:degré}, we deduce that the sequence $(\deg \phi_{\theta^n}(1))_{n\geq N}$ is strictly increasing and that for all $n\neq0$, 
$$\deg \phi_{\theta^{N+n}}(1)=q^{rn}\left(\deg \phi_{\theta^N}(1)+\dfrac{\deg a_r}{q^r-1}\right)-\dfrac{\deg a_r}{q^r-1}.$$
Thus, if $a\in A$, $\deg a=N+n$ with $n\geq 0$, then 
$$\deg \phi_a(1)=q^{rn}\left(\deg \phi_{\theta^N}(1)+\dfrac{\deg a_r}{q^r-1}\right)-\dfrac{\deg a_r}{q^r-1}$$ which is strictly greater than $N+n$ from some rank $N+n_0$, $n_0\in \N$. We then set $c_\phi=N+n_0$.
    
\end{proof}

\begin{proof}[Proof ot Theorem \ref{th:mersenne}]Assume Conjecture \ref{conj} to be true.

 Let $P\in A$ be such that $g_{P,\phi}$ is a prime number $Q$ and such that $\deg(P)>c_\phi$. In particular, $\deg Q=\deg g_{P,\phi}=\deg P>c_\phi$. Consider the element
 $$\phi_Q(1)=\phi_{g_{P,\phi}}(1)= 0 \pmod{P}.$$
 Thus, $P$ divides $\phi_Q(1)$, and by Lemma \ref{lem:degré} we have
 $$\deg \phi_Q(1)>\deg Q=\deg P.$$
Hence, the element $\phi_Q(1)$is not prime.
 \end{proof}

Note that Conjecture \ref{conj} is true for the Carlitz module. Indeed, in this case, by \cite[Theorem 3.6.3]{goss} for every $P$, we have $g_{P,\phi}=P-1$ and by a Hall theorem \cite{hall2006functions} there exist infinitely many primes $P$ such that both 
$P$ and $P-1$ are prime.

\bibliographystyle{abbrv}
\bibliography{ref}

\end{document}